\documentclass[oneside]{article}
\usepackage[english]{babel}
\usepackage{amsmath,amssymb,bbm,latexsym,theorem}
\usepackage[margin=1in]{geometry}

\newcommand{\assign}{:=}
\newcommand{\longhookrightarrow}{{\lhook\joinrel\relbar\joinrel\rightarrow}}
\newcommand{\nin}{\not\in}
\newcommand{\nobracket}{}
\newcommand{\nocomma}{}
\newcommand{\tmem}[1]{{\em #1\/}}
\newcommand{\tmop}[1]{\ensuremath{\operatorname{#1}}}
\newcommand{\tmstrong}[1]{\textbf{#1}}
\newenvironment{proof}{\noindent\textbf{Proof\ }}{\hspace*{\fill}$\Box$\medskip}
\newtheorem{corollary}{Corollary}
\newtheorem{lemma}{Lemma}
\newtheorem{proposition}{Proposition}
{\theorembodyfont{\rmfamily}\newtheorem{remark}{Remark}}
\newtheorem{theorem}{Theorem}

\begin{document}

\title{Topological Symmetries of $\mathbbm{R}^3$}

\author{Slawomir Kwasik\footnote{The author acknowledges the support of the Simons Foundation Grant No.281810} \  and Fang Sun}

\maketitle
\begin{large}

\section{Introduction}

Let $G$ be a finite group acting effectively on a smooth manifold
$\mathcal{M}^n$ . In the theory of transformation groups one usually considers
two types of actions: topological actions and smooth actions. It was realized
very early (cf. [2],[3]) that there are topological actions which can not be
smoothed out, i.e. actions which are not conjugate (topologically) to smooth
actions.

In the case of an Euclidean space $\mathbbm{R}^n$ there is a particularly nice
class of smooth actions, namely linear actions of $G$ on $\mathbbm{R}^n$.
Again there is plenty of examples (cf. [14],[15],[26]) of smooth actions which can
not be linearized (even topologically).

It turns out however that in dimension $n \leqslant 2$ every topological
finite group action on $\mathbbm{R}^n$ is conjugate to a linear action (cf.
[4],[19],[11]). For the Euclidean space $\mathbbm{R}^3$ the situation is
drastically different. First, there exist topological nonsmoothable actions of
cyclic groups on $\mathbbm{R}^3$. Second, the existence of uncountably many
nonequivalent such actions (cf. [2],[3]) dashes any hope for classification of
topological actions of finite groups on $\mathbbm{R}^3$. On the other hand
smooth actions on $\mathbbm{R}^n$, $n = 3$, are conjugate to linear actions (a
highly nontrivial result, cf. [20]), and the linearization fails for smooth
actions on $\mathbbm{R}^n, n \geqslant 4$ (cf. [14] [15],[26]). In particular it
follows that every finite group acting orientation preservingly and smoothly
on $\mathbbm{R}^n, n \leqslant 3$ is isomorphic to a subgroup of $SO
(3)$.

In fact even though there exist nonlinear orientation preserving smooth actions of finite groups on
$\mathbbm{R}^4$, such group must be isomorphic to a subgroup of $SO
(4)$ (cf. [16] and Remark 3.1 in this paper).

In a sharp contrast with smooth actions there are topological orientation
preserving actions of finite groups $G$ on $\mathbbm{R}^4$ with $G \nsubseteq{SO}
(4)$ (cf. [21] and Remark 3.2 in this paper).

Consequently the main problem left open is the question concerning the
structure of finite groups acting topologically on $\mathbbm{R}^3$. More
specifically:\\

{{\tmstrong{Problem}}: {\tmem{Let $G$ be a finite group acting
topologically and orientation preservingly on $\mathbbm{R}^3$. Is $G$
isomorphic with a finite subgroup of $SO (3)$?}}}\\

This paper provides the positive answer to this problem. Namely:

\begin{theorem}
  Let $G$ be a finite group acting topologically and orientation preservingly
  on $\mathbbm{R}^3$. Then $G$ is isomorphic with a finite subgroup of
  $SO (3)$.
\end{theorem}

It turns out that Theorem 1 extends to actions which are not necessarily orientation preserving, namely: 

\textbf{Theorem}:
  Let $G$ be a finite group acting topologically
  on $\mathbbm{R}^3$. Then $G$ is isomorphic with a finite subgroup of
  $O (3)$.

The proof of the above theorem uses heavily Theorem 1 yet it is not exactly a direct consequence of Theorem 1. Consequently, it is treated separately in [27].\\

There are two main difficulties in dealing with Theorem 1.

The first difficulty is the lack of a systematic study of topological actions
of finite groups on 3-manifolds due to the existence of ``very exotic''
actions on $\mathbbm{R}^3 (S^3)$. The most striking results in this area being
constructions of R.H.Bing ([2], [3]) of cyclic group actions on $\mathbbm{R}^3
(S^3)$ with wildly embedded (and knotted) fixed point sets.

Note that such phenomena can not happen in the smooth setting (i.e., the Smith
Conjecture).

These constructions indicate serious problems (impossibility?) with a
development of equivariant versions of some powerful tools in 3-dimensional
topology which could help to analyze the above problem. Note that in the
smooth (locally linear) setting such tools were successfully developed (cf.
[23], [10], [18]).

The second difficulty in dealing with the above problem is the low dimension
of $\mathbbm{R}^n$ i.e. $n = 3$. This low dimension prevents the use of higher
dimensional techniques (for example the surgery theory) which could be helpful
in purely topological setting (cf. [21], [22] for the case of topological
actions in dimension four).

As a consequence our proof of Theorem 1 relies heavily on a careful group
theoretic analysis of the structure of finite groups which potentially can act
on $\mathbbm{R}^3$. The techniques employed to accomplish this are those of
homological algebra and finite group theory.

The considerations leading to the proof of Theorem 1 turn out to be
surprisingly intricate and delicate.

To a large extend this is however expected. For, dealing with purely
topological actions, even on surfaces, without any assumption about
geometrization of these actions, usually requires a fair amount of work (cf.
[11], [24]).

\section{Proofs}

This section contains the proof of Theorem 1. Before providing the details we
outline the general strategy of the proof.

Throughout this section, $G$ is assumed to be a finite group acting
topologically and orientation preservingly on $\mathbbm{R}^3$. All group
actions in this paper are assumed to be effective.

\subsection{Outline of Proof}

Our considerations are divided into two cases.

{\underline{Case 1}}: $G$ is solvable.

{\underline{Case 2}}: $G$ is not solvable.

Suppose now that $G$ is solvable. We will investigate the structure of $G$ via
a subnormal series whose factors are cyclic groups of prime order (the
existence of such is guaranteed by solvability). In other words, we will treat
$G$ as obtained by a finite sequence of group extensions by $\mathbbm{Z}_p$,
$p$ prime. All possible sequences can be expressed in a ``tree diagram'' with
the trivial group at the origin, possible $G$ at the other end and each line
segment representing an extension by $\mathbbm{Z}_p$. The existence of an
action of $G$ on $\mathbbm{R}^3$ significantly reduces the possibilities at
each step. We begin with extensions of cyclic groups, and we have:

\textbf{A.1.} $\{ \nobracket 0 \longrightarrow \mathbbm{Z}_n \longrightarrow G
\longrightarrow \mathbbm{Z}_p \longrightarrow 0$, $n \in \mathbbm{Z}^+$, $p$
is odd\}$\Longrightarrow G$ is cyclic.

\textbf{A.2.} $\{ \nobracket 0 \longrightarrow \mathbbm{Z}_n \longrightarrow G
\longrightarrow \mathbbm{Z}_2 \longrightarrow 0$, $n \in \mathbbm{Z}^+ \}
\nobracket \Longrightarrow G$ is cyclic or dihedral.

Here $\mathbbm{Z}^+ = \{ 1, 2, 3, \ldots \}$ and $\mathbbm{Z}_1 = 0$.

Thus we obtain part of the tree diagram

\begin{eqnarray*}
  &  & \{ 0 \} \overset{\tmop{odd}}{\longrightarrow} \tmop{Cyclic}
  \overset{\mathbbm{Z}_2}{\longrightarrow}  \left\{ \begin{array}{l}
    \tmop{Cyclic} \overset{\tmop{odd}}{\longrightarrow} \tmop{Cyclic}
    \overset{\mathbbm{Z}_2}{\longrightarrow} \left\{ \begin{array}{l}
      \tmop{Cyclic} \overset{\tmop{odd}}{\longrightarrow} \ldots\\
      \\
      \tmop{Dihedral}
    \end{array} \right.\\
    \\
    \tmop{Dihedral}
  \end{array} \right.
\end{eqnarray*}
Where the ``odd'' step is a consecutive extension by $\mathbbm{Z}_p$, $p$ odd.

\

Next we look at branches (extensions) stemming from dihedral groups, here we
have:

\textbf{B.1.} $\{ \nobracket 0 \longrightarrow D_{2 n} \longrightarrow G
\longrightarrow \mathbbm{Z}_p \longrightarrow 0$, $n \geqslant 2$, $p$ is
odd\} $\Longrightarrow$ $G$ is dihedral or the alternating group $A_4$

\textbf{B.2.} $\{ \nobracket 0 \longrightarrow D_{2 n} \longrightarrow G \longrightarrow
\mathbbm{Z}_2 \longrightarrow 0$, $n \geqslant 2$\} $\Longrightarrow G$ is
dihedral

Consequently such branch assumes the following form:
\begin{eqnarray*}
  & \longrightarrow & \tmop{Dihedral} \overset{\tmop{odd}}{\longrightarrow} 
  \left\{ \begin{array}{l}
    \tmop{Dihedral} \overset{\mathbbm{Z}_2}{\longrightarrow} \tmop{Dihedral}
    \longrightarrow \ldots\\
    \\
    A_4
  \end{array} \right.
\end{eqnarray*}

Now for a branch stemming from $A_4$, we have:

\textbf{C.} \{$0 \longrightarrow A_4 \longrightarrow G \longrightarrow \mathbbm{Z}_p
\longrightarrow 0$, $p$ prime\}$\Longrightarrow$ $p = 2$ and $G \cong S_4$.
\\

Regarding $S_4$, we have:

\textbf{D.} There is no extension of $S_4$.
\\

Thus $G$, sitting at the end of the sequence of extensions, must be either
cyclic, dihedral, $A_4$ or $S_4$. In particular $G \subset SO (3)$.
This finishes Case 1.

\

Now suppose $G$ is not solvable.

From the above discussions we know that the Sylow $p$-subgroup of $G$,
$\tmop{Syl}_p (G)$, is cyclic for $p$-odd. For the $2$-subgroups, it turns out
$\tmop{Syl}_2 (G)$ is either cyclic or dihedral.

If $\tmop{Syl}_2 (G)$ is cyclic, then $G$ is metacyclic, in particular
solvable (cf. [17]), thus $G \subset SO (3)$.

If $\tmop{Syl}_2 (G)$ is dihedral, a theorem of Suzuki (cf. [28]) gives a
normal subgroup $G_1 \vartriangleleft G$ with $[G : G_1] \leqslant 2$ and $G
\cong Z \times L$, where $Z$ is solvable and $L \cong PSL (2, p)$,
$p$-prime. It will be shown that in this case $G$ is solvable or $G \cong
A_5$. In either case $G \subset SO (3)$.

\

This concludes the outline of the proof of Theorem 1.

Now we present the necessary details of the above process.

\

\subsection{Preliminaries}

Since $G$ acts (orientation preservingly) on $\mathbbm{R}^3$, there is an
induced orientation preserving action on $S^3$, the one point
compactification. In what follows, the action of $G$ on $S^3$ will be assumed
to be this one.

The following results will be crucial in our proof.

\begin{proposition}
  If G is cyclic and nontrivial, then $(S^3)^G = S^1$(= stands for
  homeomorphism).
\end{proposition}

\begin{proof}
  See [25].
\end{proof}

The above proposition has the following corollary:

\begin{corollary}
  If $(S^3)^G = S^1$, then G is cyclic.
\end{corollary}

\begin{proof}
  Let $H \subset G$ be a nontrivial subgroup. Take any nontrivial element $a
  \in H$, then $\langle a \rangle \subset H \subset G$, where $\langle a \rangle$ is the cyclic
  subgroup generated by $a$. Thus $(S^3)^{\langle a \rangle} \supset (S^3)^H \supset
  (S^3)^G$. The two ends of this sequence are homeomorphic to $S^1$ by
  Proposition 1 and the assumption. Since any embedding of an $S^1$ in another
  is surjective, $(S^3)^H = S^1 = (S^3)^G$. Now $H$ is taken arbitrarily,
  which implies that the action of $G$ on $S^3 - (S^3)^G = S^3 - S^1$ is free.
  It is a well known fact that $S^3 - S^1$ is a cohomological 1-sphere. Thus
  the Tate cohomology of $G$ has period 2, i.e., $G$ is cyclic (cf. [5]).
\end{proof}

Consequently for a finite group $G$ acting orientation preservingly on $R^3$,
$(S^3)^G = S^1$ if and only if $G$ is cyclic and nontrivial.

\

Another result which plays an important role in our considerations is the
following:

\begin{theorem}
  If G is a 2-group, then G is cyclic or dihedral.
\end{theorem}

\begin{proof}
  By the result of Dotzel and Hamrick (cf. [9]), there is an orthogonal action
  of $G$ on $\mathbbm{R}^4$ such that there exist a map $\phi : \Sigma^k S^3
  \longrightarrow \Sigma^k S^3, k \geqslant 1$ equivariant (with respect to
  the original action for the domain and the action induced by the orthogonal
  action on $S^3 \subset \mathbbm{R}^4$ for the codomain) which induces
  $\mathbbm{Z}_p$-homology isomorphisms on (suspension of) fixed point sets of
  non-trivial subgroups. These fixed point sets are spheres, thus their
  dimension are preserved by $\phi$.
  
  Since $(S^3)^G \neq \varnothing$ in the original action, so the same holds
  for the orthogonal action. Let $v \in S^0$ be a fixed point. Let $W$ be the
  orthogonal complement of $v$ in $\mathbbm{R}^4$. $G$ acts orthogonally on
  $W$. Checking the fixed point set, one sees that this action is faithful,
  thus $G \subset O (3)$.
  
  Now suppose $G$ is neither cyclic or dihedral, then as a subgroup of
  $O (3)$, $G$ contains $-I_3$. Let $A = -I_3$
  
  Since $ \langle A \rangle $ is cyclic, then $(S^3)^{\langle A \rangle} = S^1$ in the original action
  by Proposition 1, and hence the same is true for the orthogonal action on
  $S^3 \subset \mathbbm{R}^4$. This implies that $(\mathbbm{R}^4)^{\langle A \rangle}
  =\mathbbm{R}^2$. This is impossible. Thus $G$ has to be cyclic or dihedral.
\end{proof}

\begin{remark}
  The first two paragraphs do not need the assumption on the action to be
  orientation preserving, thus we have shown that any 2-group acting on
  $\mathbbm{R}^3$ is a subgroup of $O (3)$.
\end{remark}

\subsection{Obstruction Kernels}

Our proof of Theorem 1 relies essentially on detecting in a given group a
typical subgroup that can not act on $\mathbbm{R}^3$. These groups will be
called obstruction kernels (abbreviated as O.K.) in what follows. In this
subsection we list all of them and prove they cannot act faithfully and
orientation preservingly on $\mathbbm{R}^3$.

We start with the following lemma:

\begin{lemma}
  If G acts orientation preservingly on $\mathbbm{R}^3$, $H, H'$ are cyclic
  subgroups of G, $H \cap H' \neq \{ 0 \},$ $\langle H \cup H' \rangle = G$ where $\langle H \cup H'
  \rangle$ denotes the subgroup generated by $H \cup H'$ . Then $G$ is cyclic.
\end{lemma}

\begin{proof}
  Consider the inclusions $(S^3)^H \subset (S^3)^{H \cap H'} \supset
  (S^3)^{H'} $. Since these subgroups are all nontrivial and cyclic,
  Proposition 1 implies that $(S^3)^H = (S^3)^{H'} = S^1$. Now $G = \langle H \cup
  H' \rangle$ implies $(S^3)^G = S^1$, thus $G$ is cyclic by Corollary 1.
\end{proof}

\begin{flushleft}
{\tmstrong{Obstruction Kernel of Type 0}}: $G = (\mathbbm{Z}_p \oplus
\mathbbm{Z}_q) \rtimes_{\varphi} \mathbbm{Z}_2$ where (the generator of)
$\mathbbm{Z}_2$ acts as identity on $\mathbbm{Z}_p$, and as multiplication by
$- 1$ on $\mathbbm{Z}_q$, $p, q$ distinct odd primes.
\end{flushleft}

\begin{proof}
  There is a canonical subgroup $\mathbbm{Z}_p \rtimes_{\varphi} \mathbbm{Z}_2
  \subset (\mathbbm{Z}_p \oplus \mathbbm{Z}_q) \rtimes_{\varphi} \mathbbm{Z}_2 = G$.
  Assume such action of $G$ on $\mathbbm{R}^3$ exists. Since the action on
  $\mathbbm{Z}_p$ is trivial, $\mathbbm{Z}_p \rtimes_{\varphi} \mathbbm{Z}_2
  \cong\mathbbm{Z}_{2 p}$. Letting $H =\mathbbm{Z}_{2 p}, H' =\mathbbm{Z}_p \oplus
  \mathbbm{Z}_q =\mathbbm{Z}_{p q}$, we get $H \cap H' =\mathbbm{Z}_p$, $\langle H
  \cup H' \rangle = G$. So by Lemma 1 $G$ is cyclic. But this implies that the
  subgroup $\mathbbm{Z}_q \rtimes_{\varphi} \mathbbm{Z}_2$ is cyclic,
  contradicting to the assumption.
\end{proof}

\begin{flushleft}
{\tmstrong{Obstruction Kernel of Type 1}}: $G =\mathbbm{Z}_{2 m} \times
\mathbbm{Z}_2, m \geqslant 2$.
\end{flushleft}

\begin{proof}
  Assume the action exists. Let $H =\mathbbm{Z}_{2 m} \subset G$, $H' = \langle (1,
  1) \rangle \subset G$. Then $H, H'$ satisfy the assumption of Lemma 1, thus $G =
  \mathbbm{Z}_{2 m} \times \mathbbm{Z}_2$ is cyclic, which is impossible.
\end{proof}

\begin{flushleft}
{\tmstrong{Obstruction Kernel of Type 2}}: $G =\mathbbm{Z}_q \rtimes_{\varphi}
\mathbbm{Z}_{2^{k + 1}}, k \geqslant 1 \nocomma,$ $q$ odd prime. $\varphi (1)$
is multiplication by $(- 1)$.
\end{flushleft}

\begin{proof}
  Assume the action exists. Let $H =\mathbbm{Z}_q \rtimes_{\varphi}
  \mathbbm{Z}_{2^k} \subset\mathbbm{Z}_q \rtimes_{\varphi} \mathbbm{Z}_{2^{k + 1}}$,
  where $\mathbbm{Z}_{2^k}$ is generated by $2 \in \mathbbm{Z}_{2^{k + 1}}$.
  Then $H$ is cyclic since the action $\varphi$ restricts to a trivial action
  on $\mathbbm{Z}_{2^k}$. Let $H'$ be the subgroup $\mathbbm{Z}_{2^{k + 1}}$
  of $G$. Then $H, H'$ satisfy the condition of Lemma 1, and hence $G$ is
  cyclic, which once again is impossible.
\end{proof}

\begin{flushleft}
{\tmstrong{Obstruction Kernel of Type 3}}: Generalized quaternion group $Q_{4
m} \nocomma, (m \geqslant 2)$.
\end{flushleft}

\begin{proof}
  Assume the action exists. Let $q$ be a prime factor of $m$. Note that there
  is a canonical subgroup $Q_{4 q} \subset Q_{4 m}$.
  
  If $q$ is odd, $Q_{4 q} =\mathbbm{Z}_q \rtimes_{\varphi} \mathbbm{Z}_4$,
  $\varphi (1)$ is multiplication by $(- 1)$. This is O.K. of Type 2.
  
  If $q = 2$, then $Q_{4q} = Q_8$ is a 2-group. This case is excluded by Theorem 2. 
\end{proof}

\begin{flushleft}
{\tmstrong{Obstruction Kernel of Type 4}}: $G =\mathbbm{Z}_{2^k}
\rtimes_{\varphi} \mathbbm{Z}_2, k \geqslant 3$, $\varphi (1)$ is
multiplication by $2^{k - 1} \pm 1$.
\end{flushleft}

\begin{proof}
 This is a corollary of Theorem 2.
\end{proof}

\begin{flushleft}
{\tmstrong{Obstruction Kernel of Type 5}}: $G =\mathbbm{Z}_p \rtimes_{\varphi}
\mathbbm{Z}_4$, $p$ prime, $p \equiv 1 (\tmop{mod} 4)$, $\varphi (1) = n \in
\mathbbm{Z}_p^{\ast} = \tmop{Aut} \mathbbm{Z}_p$ where $n^2 = -1 (\tmop{mod}
p)$.
\end{flushleft}

\begin{proof}
  Assume the action exists. Then $\mathbbm{Z}_p$ is a normal subgroup of $G$,
  and $\mathbbm{Z}_4$ acts on the fixed point set $(S^3)^{\mathbbm{Z}_p} =
  S^1$ (Proposition 1). Let $f$ be the restriction of $\varphi (0, 1)$ on
  $(S^3)^{\mathbbm{Z}_p} = S^1$. Note that $f^4$, being the restriction of
  $\varphi (0, 4)$, is the identity, and a homeomorphism $f$ of $S^1$
  satisfying this must also satisfy $f^2 = \tmop{id}$. That is, $(0, 2) \in
  \mathbbm{Z}_p \rtimes_{\varphi} \mathbbm{Z}_4$ acts on
  $(S^3)^{\mathbbm{Z}_p} = S^1$ by the identity. Thus $(S^3)^{\mathbbm{Z}_p
  \rtimes_{\varphi} \mathbbm{Z}_2} = (S^3)^{\mathbbm{Z}_p} = S^1$, where
  $\mathbbm{Z}_p \rtimes_{\varphi} \mathbbm{Z}_2 \subset \mathbbm{Z}_p
  \rtimes_{\varphi} \mathbbm{Z}_4$ is the canonical subgroup. Hence
  $\mathbbm{Z}_p \rtimes_{\varphi} \mathbbm{Z}_2$ is cyclic. But the action
  $\varphi$ restricted to $\mathbbm{Z}_2 \subset \mathbbm{Z}_4$ is nontrivial
  by definition which is a contradiction.
\end{proof}

Before presenting the O.K. of Type 6, we need to address the convention
regarding the automorphism group of the dihedral group $D_{2 n}$. It is a well
known fact that $\tmop{Aut} D_{2 n} \cong \mathbbm{Z}_n \rtimes_{}
\mathbbm{Z}_n^{\ast}$ where the semi-direct product is with respect to the
canonical identification $\tmop{Aut} \mathbbm{Z}_n \cong
\mathbbm{Z}^{\ast}_n$. An element $(t, s) \in \mathbbm{Z}_n \rtimes_{}
\mathbbm{Z}_n^{\ast}$ maps $a^i b$ to $a^{s i + t} b$ and $a^i$ to $a^{s i}$.
Here $a, b$ are standard generators of $D_{2 n}$ ($a^n = b^2 = \tmop{id}, a b
a = b$).

\begin{flushleft}
{\tmstrong{Obstruction Kernel of Type 6}}: $G = D_{2 n} \rtimes_{\varphi}
\mathbbm{Z}_2$, $n$ even, $n \geqslant 4$, $\varphi (1)$ is $(t, - 1) \in
\mathbbm{Z}_n \rtimes_{} \mathbbm{Z}_n^{\ast}$, $t$ even.
\end{flushleft}

\begin{proof}
  Assume the action exists. Let $i = \frac{t + 2}{2}$, then $2 i - t = 2$. Now
  $(a^i b, 1)^2 = (a^{2 i - t}, 0) = (a^2, 0)$. Define $H = \langle (a^i b, 1) \rangle$,
  and note that $| H | = n$. Define $H' = \langle (a, 0) \rangle$ with $| H' | = n$. Then
  $H \cap H' = \langle (a^2, 0) \rangle \neq \{ 0 \}$. Let $G' = \langle H \cup H' \rangle \nocomma$,
  then Lemma 1 can be applied to $G', H, H'$. Thus $G'$ is cyclic, of order$>
  n$. But examining the elements of $G = D_{2 n} \rtimes_{\varphi}
  \mathbbm{Z}_2$, none has order exceeding $n$, a contradiction.
\end{proof}

\

\subsection{The Solvable Case}

Now we consider the case of solvable $G$. We will determine what happens for
each extension step, then an induction will produce the desired result.

\subsubsection{Extension of Cyclic Groups}

We are going to prove A.1 and A.2 of the Section 2.1. Let us start with A.1.

\begin{proposition}
  Suppose there is a short exact sequence $0 \rightarrow \mathbbm{Z}_n
  \rightarrow G \rightarrow \mathbbm{Z}_p \rightarrow 0$, $p$ odd prime. Then
  $G$ is cyclic.
\end{proposition}

\begin{proof}
  Proposition 1 implies that $(S^3)^{\mathbbm{Z}_n} = S^1$. Now
  $\mathbbm{Z}_p$ acts on this $S^1$, but $p$ is odd, so the action has to be
  trivial. Thus $(S^3)^G = ((S^3)^{\mathbbm{Z}_n})^{\mathbbm{Z}_p} = S^1$, and
  $G$ is cyclic by Corollary 1.
\end{proof}

\begin{corollary}
  If $| G |$ is odd, then $G$ is cyclic.
\end{corollary}

\begin{proof}
  This is because all odd order groups are solvable. (cf. [12]).
\end{proof}

The assertion A.2 is more delicate, and we have to discuss several cases,
divided by the number of 2-factors in the order of the cyclic group.

\begin{proposition}
  Suppose there is a short exact sequence $0 \rightarrow \mathbbm{Z}_n
  \rightarrow G \rightarrow \mathbbm{Z}_2 \rightarrow 0$, n odd. Then G is
  cyclic or dihedral.
\end{proposition}

\begin{proof}
  Let $p_1^{n_1} \ldots p_k^{n_k}$ be the prime decomposition of $n$. Since
  $(n, 2) = 1$, the sequence splits (cf. [5] p.93). Thus
  \[ G \cong \mathbbm{Z}_n \rtimes_{\varphi} \mathbbm{Z}_2, \varphi :
     \mathbbm{Z}_2 \rightarrow \tmop{Aut} \mathbbm{Z}_n = \tmop{Aut}
     \underset{i}{\Pi} \mathbbm{Z}_{p_i^{n_i}} = \underset{i}{\Pi} \tmop{Aut}
     \mathbbm{Z}_{p_i^{n_i}} \]
  The component of $\varphi (1)$ on each $\tmop{Aut} \mathbbm{Z}_{p_i^{n_i}} $
  is a multiplication by $\pm 1$ (because $\varphi (1)$ has order $\leq$ 2, and
  $\tmop{Aut} \mathbbm{Z}_{p_i^{n_i}}$ is cyclic thus have a unique element of
  order 2).
  
  If all components are $+ 1$, then $G \cong \mathbbm{Z}_n \times
  \mathbbm{Z}_2 \cong \mathbbm{Z}_{2 n}$ is a cyclic group.
  
  If all components are $- 1$, $\varphi (1)$ is multiplication of $- 1$ on
  $\mathbbm{Z}_n$, then $G \cong D_{2 n}$ is the dihedral group.
  
  If both $\pm 1$ exist, take $p$, $q$ prime factor of $n$ where $\varphi (1)$
  is multiplication of $+ 1$, $- 1$ respectively on the corresponding
  components. Now there is a canonical subgroup $G_0 \subset G$, i.e. $G_0 : =
  (\mathbbm{Z}_p \oplus \mathbbm{Z}_q) \rtimes_{\varphi} \mathbbm{Z}_2 \subset
  G$. But $G_0$ is O.K. of Type 0; a contradiction.
  
  In conclusion, the only possibility is to be a cyclic or dihedral group.
\end{proof}

\begin{proposition}
  Suppose there is a short exact sequence $0 \rightarrow \mathbbm{Z}_{2 n}
  \rightarrow G \rightarrow \mathbbm{Z}_2 \rightarrow 0$, n odd. Then G is
  cyclic or dihedral.
\end{proposition}

\begin{proof}
  Note that in the above extension, there is an induced action $\varphi$ of
  $\mathbbm{Z}_2$ on $\mathbbm{Z}_{2 n}$. As in the previous proof, $\varphi
  (1)$ is a multiplication of $\pm 1$ on each factor of the prime
  decomposition $\mathbbm{Z}_{2 n}$. For convenience, we name those prime with
  $\varphi (1)$ being $+ 1$ as $p_i, 1 \leqslant i \leqslant k$, and those
  corresponding to $- 1$ as $q_j, 1 \leqslant j \leqslant l$ (the action on the
  $\mathbbm{Z}_2$ component is always trivial). Let $P = \underset{i}{\Pi}
  p_i^{n_i}, Q = \underset{j}{\Pi} p_j^{m_j}, n = 2 P Q$. An easy computation
  shows that $H^2 (\mathbbm{Z}_2, \mathbbm{Z}_{2 n})$, the cohomology of
  $\mathbbm{Z}_2$ with coefficient $\mathbbm{Z}_{2 n}$( being a
  $\mathbbm{Z}_2$-module via $\varphi$), is $\mathbbm{Z}_2$ for any $\varphi$.
  Thus up to equivalence there are two extensions for each fixed $\varphi$.
  
  A representative from each equivalence class can be given as:
  
  Split Case: $0 \longrightarrow \mathbbm{Z}_{2 n} \longrightarrow
  \mathbbm{Z}_{2 n} \rtimes_{\varphi} \mathbbm{Z}_2 \longrightarrow
  \mathbbm{Z}_2 \longrightarrow 0$, the semi-direct product.
  
  Non-split Case: $0 \longrightarrow \mathbbm{Z}_{2 n}
  \overset{\alpha}{\longrightarrow} \mathbbm{Z}_Q \rtimes_{\phi}
  \mathbbm{Z}_{4 P} \overset{\beta}{\longrightarrow} \mathbbm{Z}_2
  \longrightarrow 0$, here $\phi (1)$ is multiplication by $- 1$ on
  $\mathbbm{Z}_Q$, $\alpha$ is induced by $\mathbbm{Z}_{2 n}
  \longleftrightarrow \mathbbm{Z}_Q \times \mathbbm{Z}_{2 P} =\mathbbm{Z}_Q
  \rtimes_{\phi} \mathbbm{Z}_{2 P} \subset\mathbbm{Z}_Q \rtimes_{\phi}
  \mathbbm{Z}_{4 P}$, $\alpha (1) = (1, 2) \in \mathbbm{Z}_Q \rtimes_{\phi}
  \mathbbm{Z}_{4 P}$, $\beta$ is the projection on the quotient.
  
  Now we investigate each case.
  
  {\underline{Split Case}}:
  
  1) If $Q = 1$, then $\varphi$ is trivial and $G =\mathbbm{Z}_{2 n} \times
  \mathbbm{Z}_2$. For $n \geqslant 2$ this is O.K. of Type 1, which is impossible. If $n = 1$ then $G$ is dihedral.
  
  2) If $P = 1$, then the action of $\varphi$ on $\mathbbm{Z}_{2 n}$ is
  multiplication by $- 1$, so $G$ is the dihedral group $D_{4 n}$.
  
  3) If neither $P$ or $Q$ is 1, then taking primes $p, q$ from their
  respective family we obtain again $(\mathbbm{Z}_p \oplus \mathbbm{Z}_q)
  \rtimes_{\varphi} \mathbbm{Z}_2 \subset G$ which is O.K. of Type 0, and
  hence this case cannnot occur.
  
  {\underline{Non-split Case}}.
  
  1) If $Q = 1$, $G \cong \mathbbm{Z}_{4 P}$, cyclic.
  
  2) If $Q > 1$, then take one prime factor $q$ from its family. There is a
  canonical subgroup $\mathbbm{Z}_q \rtimes_{\phi} \mathbbm{Z}_4 \subset
  \mathbbm{Z}_Q \rtimes_{\phi} \mathbbm{Z}_{4 P} = G$. Since $P$ is odd, $\phi
  (1)$ acts on $\mathbbm{Z}_q$ by $- 1$. This is O.K. of Type 2, so cannot
  occur.
  
  In conclusion, $G$ is cyclic or dihedral.
\end{proof}

\begin{proposition}
  Suppose there is a short exact sequence $0 \rightarrow \mathbbm{Z}_{2^k n}
  \rightarrow G \rightarrow \mathbbm{Z}_2 \rightarrow 0$, n odd, $k \geqslant
  2$. Then G is cyclic or dihedral.
\end{proposition}

\begin{proof}
  Again there is an induced action of $\mathbbm{Z}_2$ on $\mathbbm{Z}_{2^k
  n}$,
  \[ \varphi : \mathbbm{Z}_2 \longrightarrow \tmop{Aut} \mathbbm{Z}_{2^k n} =
     \underset{i}{\Pi} \tmop{Aut} \mathbbm{Z}_{p_i^{n_i}} \times \tmop{Aut}
     \mathbbm{Z}_{2^k} \]
  where $n = 2^k p_1^{n_1} \ldots p_l^{n_l}$ is the prime decomposition of
  $n$. As before we rename the $p_i$'s as $p_i$ and $q_j$ according to the
  sign of $\varphi (1)$ on the corresponding components, and define $P =
  \underset{i}{\Pi} p_i, Q = \underset{j}{\Pi} q_j$.
  
  Now for fixed $n, P, Q, k$, the possible $\varphi$'s are classified by their
  actions on $\mathbbm{Z}_{2^k}$. Since $\tmop{Aut} \mathbbm{Z}_{2^k} \cong
  \mathbbm{Z}_2 \oplus \mathbbm{Z}_{2^{k - 2}}$, there is no more than four
  elements of order $\leq$ 2 (when $k = 2$ there are two). It is not hard to see the
  possibilities for $\varphi (1)$ are the following multiplications on
  $\mathbbm{Z}_{2^k}$:
  
  1)$+ 1$, in which case $H^2 (\mathbbm{Z}_2, \mathbbm{Z}_{2^k n})
  =\mathbbm{Z}_2$.
  
  2)$- 1$, in which case $H^2 (\mathbbm{Z}_2, \mathbbm{Z}_{2^k n})
  =\mathbbm{Z}_2$.
  
  3)$2^{k - 1} + 1$, ($k > 2$), in which case $H^2 (\mathbbm{Z}_2,
  \mathbbm{Z}_{2^k n}) = 0$.
  
  4)$2^{k - 1} - 1$, ($k > 2$), in which case $H^2 (\mathbbm{Z}_2,
  \mathbbm{Z}_{2^k n}) = 0$.
  
  Therefore the equivalence classes of extensions in each case are:
  
  1)There are two equivalence classes:
  
  {\underline{Split Case}}: $0 \longrightarrow \mathbbm{Z}_{2^k n}
  \longrightarrow \mathbbm{Z}_{2^k n} \rtimes_{\varphi} \mathbbm{Z}_2
  \longrightarrow \mathbbm{Z}_2 \longrightarrow 0$, the semidirect product.
  Note that $G$ contains $\mathbbm{Z}_2^k \rtimes \mathbbm{Z}_2 \cong
  \mathbbm{Z}_{2^k} \times \mathbbm{Z}_2$ since $\varphi$ is trivial on
  $\mathbbm{Z}_{2^k}^{}$. But this is O.K. of Type 1; impossible.
  
  {\underline{Non-split Case}}: $0 \longrightarrow \mathbbm{Z}_{2^k n}
  \longrightarrow (\mathbbm{Z}_P \oplus \mathbbm{Z}_Q) \rtimes_{\phi}
  \mathbbm{Z}_{2^{k + 1}} \longrightarrow \mathbbm{Z}_2 \longrightarrow 0
  \nocomma, \phi (1)$ is multiplication by $+ 1, - 1$ on $\mathbbm{Z}_P$,
  $\mathbbm{Z}_Q$ respectively, and $\mathbbm{Z}_{2^k n} \cong \mathbbm{Z}_P
  \times \mathbbm{Z}_Q \times \mathbbm{Z}_{2^k_{}} \longhookrightarrow
  \mathbbm{Z}_P \times \mathbbm{Z}_Q \rtimes_{\phi} \mathbbm{Z}_{2^{k +
  1}_{}}$ canonically. This extension induces $\varphi$ and does not split.
  The group $G$ contains an O.K. of Type 2 $\mathbbm{Z}_q \rtimes
  \mathbbm{Z}_{2^{^{k + 1}}}$ unless $Q = 1$. So the only possible $G$ here is
  $\mathbbm{Z}_P \rtimes_{\phi} \mathbbm{Z}_{2^{k + 1}} \cong \mathbbm{Z}_P
  \times \mathbbm{Z}_{2^{k + 1}}^{}$, which is cyclic.
  
  2)Again there are two equivalence classes:
  
  {\underline{Split Case}}: $0 \longrightarrow \mathbbm{Z}_{2^k n}
  \longrightarrow \mathbbm{Z}_{2^k n} \rtimes_{\varphi} \mathbbm{Z}_2
  \longrightarrow \mathbbm{Z}_2 \longrightarrow 0$. The group $G =
  (\mathbbm{Z}_P \oplus \mathbbm{Z}_Q \oplus \mathbbm{Z}_{2^k})
  \rtimes_{\varphi} \mathbbm{Z}_2$ contains $(\mathbbm{Z}_P \oplus
  \mathbbm{Z}_2) \rtimes_{\varphi} \mathbbm{Z}_2$ ($\mathbbm{Z}_P \oplus
  \mathbbm{Z}_2$ is the fixed point set of the action by $\varphi$) unless $P
  = 1$. But $(\mathbbm{Z}_P \oplus \mathbbm{Z}_2) \rtimes_{\varphi}
  \mathbbm{Z}_2 \cong \mathbbm{Z}_{2 P} \rtimes \mathbbm{Z}_2 \cong
  \mathbbm{Z}_{2 P} \times \mathbbm{Z}_2$, i.e. O.K. of Type 1. So $P = 1$,
  and $G = (\mathbbm{Z}_Q \oplus \mathbbm{Z}_{2^k}) \rtimes_{\varphi}
  \mathbbm{Z}_2$, $\varphi (1)$ acts as multiplication by $- 1$, so $G$ is
  dihedral.
  
  {\underline{Non-split Case}}: $0 \longrightarrow \mathbbm{Z}_{2^k n}
  \overset{\alpha}{\longrightarrow} (\mathbbm{Z}_P \oplus \mathbbm{Z}_Q)
  \rtimes_{\phi} Q_{4 m} \overset{\beta}{\longrightarrow} \mathbbm{Z}_2
  \longrightarrow 0 \nocomma,$ where $m = 2^{k - 1} \geqslant 2$. Here
  $\phi$ is defined as follows. Let $x = e^{\frac{i \pi}{m}}, y = j$ in
  $Q_{4 m}$, define $\varphi (x) = \tmop{id}, \varphi (y)$ be multiplication
  by $+ 1, - 1$ on $\mathbbm{Z}_P, \mathbbm{Z}_Q$ respectively (Checking
  relations of $Q_{4 m}$ we see this is well-defined). The embedding $\alpha$
  \ comes from $\mathbbm{Z}_{2^k n} \cong \mathbbm{Z}_P \oplus \mathbbm{Z}_Q
  \oplus \mathbbm{Z}_{2^k} \longleftrightarrow \mathbbm{Z}_P \oplus
  \mathbbm{Z}_Q \rtimes_{\phi} \langle x \rangle \subset (\mathbbm{Z}_P \oplus
  \mathbbm{Z}_Q) \rtimes_{\phi} Q_{4 m}$, and $\beta$ is projection on the
  quotient. It can be readily verified that this extension is non-split and
  realizes $\varphi$. Now $G$ contains $Q_{4 m}$, which is O.K. of Type 3;
  impossible.
  
  3) Since $H^2 (\mathbbm{Z}_2, \mathbbm{Z}_{2^k n}) = 0$, there is up to
  equivalence only one extension, the semi-direct product $(\mathbbm{Z}_P
  \oplus \mathbbm{Z}_Q \oplus \mathbbm{Z}_{2^k}) \rtimes_{\varphi}
  \mathbbm{Z}_2$. This group contains O.K. of Type 4.
  
  4)Again there is only the semi-direct product. Now $G = (\mathbbm{Z}_P \oplus
  \mathbbm{Z}_Q \oplus \mathbbm{Z}_{2^k}) \rtimes_{\varphi} \mathbbm{Z}_2$
  contains $\mathbbm{Z}_{2^k} \rtimes_{\varphi} \mathbbm{Z}_2$. Checking
  the definition of $\varphi$, we see that this subgroup is precisely O.K. of
  Type 4; a contradiction.
  
  In conclusion: $G$ has to be cyclic or dihedral.
\end{proof}

Summing up Propositions 2-5, we obtain the desired result:

\begin{theorem}
  Suppose there is an extension $0 \longrightarrow \mathbbm{Z}_n
  \longrightarrow G \longrightarrow \mathbbm{Z}_p \longrightarrow 0$, $n
  \geqslant 1$, $p$ prime. Then $G$ is cyclic or dihedral.
\end{theorem}

\

\subsubsection{Extension of Dihedral Groups}

We will prove B.1 and B.2 of the outline. We start with an extension by odd
primes.

\begin{proposition}
  Suppose there is an extension $0 \longrightarrow D_4 \longrightarrow G
  \longrightarrow \mathbbm{Z}_p \longrightarrow 0$, $p$ odd prime. Then $p =
  3$ and $G \cong A_4$.
\end{proposition}

\begin{proof}
  Since $(4, p) = 1$, the sequence splits and $G = D_4 \rtimes_{\varphi}
  \mathbbm{Z}_p$ for some $\varphi : \mathbbm{Z}_p \longrightarrow \tmop{Aut}
  D_4 \cong S_3$ (here $S_3$ is the permutation group of the nontrivial
  elements of $D_4$).
  
  If $p > 3$, then $\varphi$ is trivial and $G = D_4 \times \mathbbm{Z}_p
  \cong \mathbbm{Z}_{2 p} \times \mathbbm{Z}_2$, which is an O.K. of Type 1.
  
  Thus $p = 3$. We observe that $\varphi$ cannot be trivial, for in that case
  we would have again $G = D_4 \times \mathbbm{Z}_3 \cong \mathbbm{Z}_6 \times
  \mathbbm{Z}_2$; impossible. Thus $\varphi (1)$ is a 3-cycle in $\tmop{Aut}
  D_4 = S_3$. This $G = D_4 \rtimes_{\varphi} \mathbbm{Z}_p$ is precisely
  $A_4$.
\end{proof}

\begin{proposition}
  Suppose there is an extension $0 \longrightarrow D_{2 n} \longrightarrow G
  \longrightarrow \mathbbm{Z}_p \longrightarrow 0$, $n > 2$, $p$ odd prime.
  Then $G$ is dihedral.
\end{proposition}

\begin{proof}
  If $p \nmid n$, then the sequence splits and $G = D_{2 n} \rtimes_{\varphi}
  \mathbbm{Z}_p$. There is a subgroup $\mathbbm{Z}_n \rtimes_{\varphi}
  \mathbbm{Z}_p \subset G$ (because any automorphism of $D_{2 n}$ preserves
  the $\mathbbm{Z}_n$ in $D_{2 n}$). But $\mathbbm{Z}_n \rtimes_{\varphi}
  \mathbbm{Z}_p \cong \mathbbm{Z}_{n p}$ since it is an extension of
  $\mathbbm{Z}_n$ by $\mathbbm{Z}_p$ and we have proven this in Proposition 2.
  Consequently $G$ is an extension of this $\mathbbm{Z}_{n p}$ by
  $\mathbbm{Z}_2$ and thus has to be dihedral by Theorem 3.
  
  If $p|n$, let $n = p^k l, (p, l) = 1$. There are subgroups
  $\mathbbm{Z}_{p^k} \subset\mathbbm{Z}_n \subset D_{2 n} \subset G$. Let the
  standard generators of $D_{2 n}$ be $a, b$ as mentioned before. Let $P$ be a
  $p$-Sylow subgroup containing $\mathbbm{Z}_{p^k}$. Take any $x \in P - D_{2
  n}$. Suppose $x^{- 1} a x \nin \mathbbm{Z}_n$, then $x^{- 1} a x = a^r b$
  and $(a^r b)^l = x^{- 1} a^l x = a^l$, a contradiction. Thus $x^{- 1} a x
  \in \mathbbm{Z}_n$, which implies $N_G (\mathbbm{Z}_n) = G$, which implies
  $\mathbbm{Z}_n \vartriangleleft G$. Let $Q = \langle\mathbbm{Z}_n \cup P \rangle$. Now
  $P\nsubseteq D_{2 n}$, and $[G : \mathbbm{Z}_n] = 2 p$, and hence $[G ; Q] \leqslant
  2$. But $G = Q$ implies $D_{2 n} \subset P$; impossible. Thus $[G : Q] = 2$.
  Now $\mathbbm{Z}_n \vartriangleleft Q, [Q : \mathbbm{Z}_n] = p$, i.e. $Q$ is
  $\mathbbm{Z}_n$ extended by a $\mathbbm{Z}_p$. By Proposition 2 it has to be
  cyclic. As a result, $G$ is diheral by Theorem 3 (it is an extension of $Q$
  by $\mathbbm{Z}_2$).
\end{proof}

Next we turn to extension by $\mathbbm{Z}_2$. The discussion will be divided
into two cases depending whether $n$ is odd or even.

\begin{proposition}
  Suppose there is an extension $0 \longrightarrow D_{2 n} \longrightarrow G
  \longrightarrow \mathbbm{Z}_2 \longrightarrow 0$, $n \geqslant 2$ odd. Then
  $G$ is dihedral.
\end{proposition}

\begin{proof}
  Let $n = p_1^{n_1} \ldots p_k^{n_k}$ be its prime decomposition.
  
  Since $\tmop{Aut} D_{2 n} \cong \mathbbm{Z}_n \rtimes
  \mathbbm{Z}_n^{\ast}$, and the inner automorphism \ $\tmop{Inn} D_{2 n}
  \cong \mathbbm{Z}_n \rtimes \{ \pm 1 \}$, as a result $\tmop{Out} D_{2 n} =
  \tmop{Aut} D_{2 n} / \tmop{Inn} D_{2 n} \cong \mathbbm{Z}_n^{\ast} / \{ \pm
  1 \}$.
  
  Since center of $D_{2 n}$, denoted by $C$, is trivial, then $H^3
  (\mathbbm{Z}_2, C) = 0$, $H^2 (\mathbbm{Z}_2, C) = 0$. In other words, each
  abstract kernel $\psi : \mathbbm{Z}_2 \longrightarrow \tmop{Out} D_{2 n}$
  has a unique extension up to equivalence. We will construct explicitly an
  extension for each given kernel.
  
  Let $\psi$ be given, then
  \[ \mathbbm{Z}_n^{\ast} / \{ \pm 1 \} 
       =\mathbbm{Z}_{p_1^{n_1}}^{\ast}
     \times \ldots \times \mathbbm{Z}_{p_k^{n_k}}^{\ast} / \{ \pm (1, \ldots, 1) \}\]
      \[ \cong\
      \mathbbm{Z}_{(p_1 - 1) p_1^{n_1 - 1}} \times \ldots \times
     \mathbbm{Z}_{(p_k - 1) p_k^{n_k - 1}} / \langle \left( \frac{p_1 - 1}{2}
     p_1^{n_1 - 1}, \ldots, \frac{p_k - 1}{2} p_k^{n_k - 1} \right) \rangle .\]
  Denote by $[(b_1, \ldots, b_k)]$ the element corresponding to $\psi (1)$ in
  the middle group, and by $[(a_1, \ldots, a_k)]$ the element corresponding to
  $\psi (1)$ in the bottom group, where $b_i \in \mathbbm{Z}_{p_i^{n_i}},
  a_i \in \mathbbm{Z}_{(p_i - 1) p_i^{n_i - 1}}$.
  
  Since $\psi (1)$ has order 2, $(2 a_1, \ldots, 2 a_k) = 0$ or $(2 a_1,
  \ldots, 2 a_k) = \left( \frac{p_1 - 1}{2} p_1^{n_1 - 1}, \ldots, \frac{p_k -
  1}{2} p_k^{n_k - 1} \right)$.
  
  {\underline{Case 1}}: If $(2 a_1, \ldots, 2 a_k) = 0$, then $b_i = \pm 1$.
  Define
  \[ \varphi : \mathbbm{Z}_2 \longrightarrow \mathbbm{Z}_{p_1^{n_1}}^{\ast}
     \times \ldots \times \mathbbm{Z}_{p_k^{n_k}}^{\ast} =\mathbbm{Z}_n^{\ast}
     \subset\mathbbm{Z}_n \rtimes \mathbbm{Z}_n^{\ast} \cong \tmop{Aut} D_{2 n},
     \varphi (1) = (b_1, \ldots, b_k) \]
  This is a well-defined homomorphism, and $0 \longrightarrow D_{2 n}
  \longrightarrow D_{2 n} \rtimes_{\varphi} \mathbbm{Z}_2 \longrightarrow
  \mathbbm{Z}_2 \longrightarrow 0$ induces the abstract kernel $\psi$. Hence
  $G = D_{2 n} \rtimes_{\varphi} \mathbbm{Z}_2$.
  
  Recall that $a, b$ denote standard generators of $D_{2 n}$. The action
  $\varphi$ is trivial on $\langle b \rangle$, whence $\langle b \rangle \rtimes \mathbbm{Z}_2 \cong
  \mathbbm{Z}_2 \oplus \mathbbm{Z}_2$. Now $\mathbbm{Z}_n \vartriangleleft G$,
  and $G$ is easily seen to be the inner semi-direct product of this copy of
  $\mathbbm{Z}_2 \oplus \mathbbm{Z}_2$ with $\mathbbm{Z}_n$. To be precise
  (checking induced inner automorphisms), $G \cong \mathbbm{Z}_n
  \rtimes_{\phi} (\mathbbm{Z}_2 \oplus \mathbbm{Z}_2)$ where $\varphi (1, 0)$
  is multiplication by $(- 1)$, $\varphi (0, 1)$ is determined by $(b_1,
  \ldots, b_k) \in \mathbbm{Z}_{p_1^{n_1}}^{\ast} \times \ldots \times
  \mathbbm{Z}_{p_k^{n_k}}^{\ast} \cong \mathbbm{Z}_n$.
  
  Now $b_i = \pm 1$, and there are three subcases:
  
  i)If both $\pm 1$ exist among them, take the semi-direct product of
  $\mathbbm{Z}_n$ with the second $\mathbbm{Z}_2$ component, we see there is
  an O.K. of Type 0 within as before. This case is thus excluded.
  
  ii)If $b_i = + 1$ for all $i$, take again the semi-direct product with the
  second $\mathbbm{Z}_2$, $\mathbbm{Z}_{2 n} \cong \mathbbm{Z}_n \rtimes
  \mathbbm{Z}_2 \subset \mathbbm{Z}_n \rtimes_{\phi} (\mathbbm{Z}_2 \oplus
  \mathbbm{Z}_2)$ because the action restricts to trivial action on
  $\mathbbm{Z}_n$ and $n$ is odd. Since $G$ contains a cyclic subgroup of
  index 2, $G$ has to be dihedral by Theorem 3.
  
  iii)If $b_i = - 1$ for all $i$, then $(\mathbbm{Z}_n \rtimes \langle (1, 1) \rangle)
  \subset \mathbbm{Z}_n \rtimes_{\phi} (\mathbbm{Z}_2 \oplus \mathbbm{Z}_2)$.
  Now that $\mathbbm{Z}_n \rtimes \langle (1, 1) \rangle \cong \mathbbm{Z}_{2 n}$ since
  the action is trivial. And conjugation of either copy of $\mathbbm{Z}_2$ on $\mathbbm{Z}_2n$ is multiplication by $-1$. Thus $G$ is dihedral.
  
  This concludes Case 1.
  
  {\underline{Case 2}}: If $(2 a_1, \ldots, 2 a_k) = \left( \frac{p_1 - 1}{2}
  p_1^{n_1 - 1}, \ldots, \frac{p_k - 1}{2} p_k^{n_k - 1} \right)$, all $p_i
  \equiv 1 (\tmop{mod} 4)$ since $2 | \frac{p_i - 1}{2}$. $b_i = m_i$ where
  $m_i^2 = - 1$ in $\mathbbm{Z}_{p_i^{n_i}}^{\ast}$
  
  Let $G =\mathbbm{Z}_n \rtimes_{\varphi} \mathbbm{Z}_4$, where $\varphi :
  \mathbbm{Z}_4 \longrightarrow \mathbbm{Z}_n^{\ast} = \underset{i}{\Pi}
  \mathbbm{Z}_{p_i^{n_i}}^{\ast}$ sends $1$ to $\underset{i}{\Pi} m_i$
  
  Restricting $\varphi$ to $\mathbbm{Z}_2 \subset\mathbbm{Z}_4$, the resulted
  semi-direct product is precisely $D_{2 n}$. The extension
  \[ 0 \longrightarrow D_{2 n} =\mathbbm{Z}_n \rtimes_{\varphi} \mathbbm{Z}_2
     \longrightarrow \mathbbm{Z}_n \rtimes_{\varphi} \mathbbm{Z}_4
     \longrightarrow \mathbbm{Z}_2 \longrightarrow 0 \]
  induces the abstract kernel $\psi$, and up to equivalence it is the only
  one. This extension however can not exist, since $G$ contains
  $\mathbbm{Z}_{p_1} \rtimes_{\varphi} \mathbbm{Z}_4$ where $\varphi (1) = m_1
  \in \mathbbm{Z}_p^{\ast}, m_1^2 = - 1$, an O.K. of Type 5.
  
  This concludes Case 2.
  
  In conclusion, $G$ has to be dihedral.
  
\end{proof}

\begin{proposition}
  Suppose there is an extension $0 \longrightarrow D_{2 n} \longrightarrow G
  \longrightarrow \mathbbm{Z}_2 \longrightarrow 0$, n even. Then $G$ is
  dihedral.
\end{proposition}

\begin{proof}
  If $n = 2$, then $G$ is a 2-group and hence must be dihedral by Theorem 2.
  
  Let $n > 2$ and let $n = 2^l p_1^{n_1} \ldots p_k^{n_k}$ be the prime
  decomposition. It turns out that the sequence splits. To see this, take a
  Sylow-2 subgroup of $D_{2 n}$. It has to be of the form $D_{2^{l + 1}}$. Let
  $P$ be a Sylow-2 subgroup of $G$ containing $D_{2^{l + 1}}$, $[P : D_{2^{l +
  1}}] = 2$. By Theorem 2, $P$ has to be dihedral, in particular, there exist
  $c \in P - D_{2^{l + 1}}$ of order 2. Note that $c \nin D_{2 n}$, since
  otherwise $P \subset D_{2 n}$ which is impossible. Now the map
  $\mathbbm{Z}_2 \rightarrow G$ sending $1$ to $c$ is a splitting.
  
  As a consequence $G \cong D_{2 n} \rtimes_{\varphi} \mathbbm{Z}_2$ for some
  $\varphi : \mathbbm{Z}_2 \longrightarrow \tmop{Aut} D_{2 n} \cong
  \mathbbm{Z}_n \rtimes \mathbbm{Z}_n^{\ast}$. Let $\varphi (1) = (t, s)$.
  
  There is a canonical subgroup $\mathbbm{Z}_n \rtimes_{\varphi} \mathbbm{Z}_2
  \subset D_{2 n} \rtimes_{\varphi} \mathbbm{Z}_2$. Since $\mathbbm{Z}_n
  \rtimes_{\varphi} \mathbbm{Z}_2$ is an extension of $\mathbbm{Z}_n$ by
  $\mathbbm{Z}_2$ then it is cyclic or dihedral by previous results
  (Theorem 3). We discuss these two cases separately:
  
  i)Suppose $\mathbbm{Z}_n \rtimes_{\varphi} \mathbbm{Z}_2$ is cyclic. Then
  $G$ is an extension of a cyclic group by $\mathbbm{Z}_2$, so $G$ is
  dihedral.
  
  ii)Suppose $\mathbbm{Z}_n \rtimes_{\varphi} \mathbbm{Z}_2$ is dihedral. Then
  $\varphi (1) \in \tmop{Aut} \mathbbm{Z}_n$ is multiplication by $- 1$, and
  whence $s = - 1$ (see remark previous to O.K. of Type 6). Now we divide the
  discussion by the parity of $t$:
  
  1)If $t$ is odd, take $i = \frac{t + 1}{2}$. Then $2 i - t = 1
  $ which implies $ (a^i b, 1)^2 = (a^{2 i - t}, 0) = (a, 0)$; thus $(a^i b, 1)$
  is of order $2 n$. Thus $G$ contains a cyclic subgroup of order $2 n$. It is not hard to see that $G$ is the semidirect product of this cyclic group with $\langle (b,0) \rangle$. An easy computation shows that $(b,0)(a^i b,1)(b,0)=(a^i b,1)^{-1}$, whence $G$ is dihedral.
  
  2)If $t$ is even, then $G$ is O.K. of Type 6; impossible.
  
  In conclusion, $G$ is dihedral.
\end{proof}

Summing up the results in this subsection, we have:

\begin{theorem}
  Suppose there is an extension $0 \longrightarrow D_{2 n} \longrightarrow G
  \longrightarrow \mathbbm{Z}_p \longrightarrow 0, n \geqslant 2, p$ prime.
  Then $G$ is dihedral or $A_4$.
\end{theorem}

\subsubsection{Extension of $A_4$}

There is an isomorphism $\tmop{Aut} A_4 \cong S_4$, induced by the conjugation
of elements of $S_4$ on the normal subgroup $A_4$. Let $D_4$ denote the unique
order 4 subgroup of $A_4$. $D_4 \vartriangleleft A_4$, and any automorphism of
$A_4$ preserves $D_4$. Now we are ready the state and prove:

\begin{theorem}
  Suppose there is an extension $0 \longrightarrow A_4 \longrightarrow G
  \longrightarrow \mathbbm{Z}_p \longrightarrow 0$, p prime. Then $p = 2$ and
  $G \cong S_4$
\end{theorem}

\begin{proof}
  The conjugation of any element of $G$ restricts to an automorphism of $A_4$.
  By the discussion above, such an automorphism has to preserve $D_4$, so $D_4
  \vartriangleleft G$.
  
  i)If $p > 3$, $(12, p) = 1$ then the sequence splits, so $G = A_4
  \rtimes_{\varphi} \mathbbm{Z}_p$ for some $\varphi$. In fact $\varphi$ has
  to be trivial since $| \tmop{Aut} A_4 | = | S_4 | = 24$ and $p > 3$. Now $G$
  contains $D_4 \rtimes_{\varphi} \mathbbm{Z}_p = D_4 \times \mathbbm{Z}_p
  \cong \mathbbm{Z}_{2 p} \oplus \mathbbm{Z}_2$. This is O.K. of Type 1;
  impossible.
  
  ii)If $p = 3$, let $P$ be the Sylow-3 subgroup. By Corollary 2, $P \cong
  \mathbbm{Z}_{3^n}$ for some $n$. Since $| D_4 | = 4$, $| P | = 9$, $D_4
  \vartriangleleft G$, we have $G \cong \mathbbm{D}_4 \rtimes_{\varphi}
  \mathbbm{Z}_9$ for some $\varphi$. $\varphi (1)^9 = \varphi (9) = \tmop{id}
  \in \tmop{Aut} D_4 \cong S_3$. $| S_3 | = 6$ implies $\varphi (1)^3 =
  \tmop{id}$. Consequently there is a subgroup $D_4 \times \mathbbm{Z}_3
  \subset \mathbbm{D}_4 \rtimes_{\varphi} \mathbbm{Z}_9$. Now $D_4 \times
  \mathbbm{Z}_3 \cong \mathbbm{Z}_6 \times \mathbbm{Z}_2$, an O.K. of Type 1.
  
  iii)If $p = 2$, define $P$ as the Sylow-2 subgroup of $G$ containing $D_4$.
  The group $P$ is diheral by Theorem 2, thus there is $c \in P - D_4$ with
  $c^2 = \tmop{id}$. Obviously $c \nin A_4$, so we obtain a splitting.
  Therefore $G \cong A_4 \rtimes_{\varphi} \mathbbm{Z}_2$ for some $\varphi$,
  where $\varphi (1) \in \tmop{Aut} A_4 \cong S_4$ is of order 2. We consider
  two cases:
  
  {\underline{Case 1}}: If $\varphi (1) \in A_4 \subset S_4$, then $\varphi
  (1) \in D_4$ implies $\varphi (1)_{|D_4} = \tmop{id}$ ($\varphi (1)$ is
  conjugation and $D_4$ is abelian). This however implies that the subgroup
  $D_4 \rtimes_{\varphi} \mathbbm{Z}_2 \cong \mathbbm{Z}_2 \oplus
  \mathbbm{Z}_2 \oplus \mathbbm{Z}_2$. This is impossible by Theorem 2, so
  this case is excluded.
  
  {\underline{Case 2}}: If $\varphi (1) \nin A_4$. Consider the canonical $0
  \longrightarrow A_4 \longrightarrow S_4 \longrightarrow \mathbbm{Z}_2
  \longrightarrow 0$. $\varphi : \mathbbm{Z}_2 \rightarrow S_4$ is a
  splitting. Thus $S_4 \cong A_4 \rtimes_{\varphi} \mathbbm{Z}_2 \cong G$.
  
  This proves the theorem.
\end{proof}

\subsubsection{Extension of $S_4$}

\begin{theorem}
  There is no extension of $S_4$
\end{theorem}

\begin{proof}
  Suppose there is an extension $0 \longrightarrow S_4 \longrightarrow G
  \longrightarrow \mathbbm{Z}_p \longrightarrow 0$, $p$ prime. Let $C$ be the
  center of $S_4$, $C = \{ 0 \}$. Now $\tmop{Aut} S_4 = \tmop{Inn} S_4 \cong
  S_4$, the second isomorphism being the obvious one. Thus $\tmop{Out} S_4$ is
  trivial and there exist only one abstract kernel: the trivial homomorphism
  $\mathbbm{Z}_p \rightarrow \tmop{Out} S_4$. The product extension $0
  \longrightarrow S_4 \longrightarrow S_4 \times \mathbbm{Z}_p \longrightarrow
  \mathbbm{Z}_p \longrightarrow 0$, and the fact $H^2 (\mathbbm{Z}_p, C) = 0$
  imply that this extension is the only one up to equivalence. Thus $G = S_4
  \times \mathbbm{Z}_p$.
  
  If $p = 2$, then $G$ contains $D_4 \times \mathbbm{Z}_2 \cong \mathbbm{Z}_2
  \oplus \mathbbm{Z}_2 \oplus \mathbbm{Z}_2$; impossible by Theorem 2.
  
  If $p > 2$, then $G$ contains $D_4 \times \mathbbm{Z}_p \cong \mathbbm{Z}_{2
  p} \oplus \mathbbm{Z}_2$, which is O.K. of Type 1; impossible.
\end{proof}

\subsubsection{Conclusion}

Summing up the previous subsections and using an inductive argument, we
obtain:

\begin{corollary}
  If $G$ is solvable, then $G \subset SO (3)$.
\end{corollary}

\

\subsection{Arbitrary Finite Group}

Now we can prove Theorem 1:

\begin{proof}
  Our first observation is that by previous discussions (Theorem 2 and
  Corollary 2), the Sylow-$p$ subgroups of $G$ are cyclic for odd $p$, and
  cyclic or dihedral for $p = 2$. If the Sylow-2 subgroups are cyclic, then
  $G$ is solvable (cf. [17] p.143), and Corollary 3 takes care of this
  situation. Thus it suffices to assume that Sylow-2 subgroups of $G$ are
  dihedral.
  
  By the result of Suzuki (cf. [28]), there exist a subgroup $G_1$ of $G$
  having the following properties: $G_1 \vartriangleleft G$, $[G : G_1]
  \leqslant 2$, $G_1 \cong Z \times L$ where $Z$ is a solvable group and $L =
  PSL (2, p)$ the projective linear group for prime $p$.
  
  i)If $p = 2$ or 3, then $L \cong S_3$ or $L \cong A_4$, in either case $G_1$
  is solvable, thus so is $G$. Corollary 3 gives the desired result.
  
  ii)If $p = 5$, then $L = A_5$. $Z$ has to be trivial. This is because
  otherwise it contains a copy of $\mathbbm{Z}_q$, $q$ prime. This
  $\mathbbm{Z}_q$ together with $D_4 \subset A_5$ produce $\mathbbm{Z}_q
  \times D_4 \subset Z \times L = G_1$. But $\mathbbm{Z}_q \times D_4 \cong
  \mathbbm{Z}_{2 q} \oplus \mathbbm{Z}_2$, an O.K. of Type 1. So $Z = 1$, $G_1
  \cong A_5$.
  
  Suppose $[G : G_1] = 2$. We have an extension $0 \longrightarrow A_5
  \longrightarrow G \longrightarrow \mathbbm{Z}_2 \longrightarrow 0$. Now
  $\tmop{Aut} A_5 = S_5,$ and $\tmop{Inn} A_5 = A_5$. Hence there are only two
  possible abstract kernels $\mathbbm{Z}_2 \rightarrow \tmop{Out} A_5 \cong
  \mathbbm{Z}_2$. Since $A_5$ is centerless ($C = \{ 0 \}$), then $H^2
  (\mathbbm{Z}_2, C) = 0$ in either case. In other words, there are only two
  extensions up to isomorphism. Now $G \cong S_5$ and $G \cong A_5 \times
  \mathbbm{Z}_2$ (together with obvious short exact sequences) are
  nonequivalent extensions giving all possibilities.
  
  {\underline{Case 1}}: $G \cong S_5$. In this case $G$ contains a general
  affine group $GA (1, 5)$ (e.g. the subgroup $\langle (12345), (2354) \rangle$ of $S_5$).
  Note that $GA (1, 5) \cong \mathbbm{Z}_5 \rtimes
  \mathbbm{Z}_5^{\ast}$., where the semidirect product is the canonical one. This subgroup
  is solvable. By Corollary 3, $\mathbbm{Z}_5 \rtimes \mathbbm{Z}_5^{\ast}$ is
  isomorphic to either a cyclic/dihedral group, $A_4$ or $S_4$. This is
  however impossible (for the dihedral case it is convenient to use the fact
  $\mathbbm{Z}_5^{\ast} \cong \mathbbm{Z}_4$).
  
  {\underline{Case 2}}: $G \cong A_5 \times \mathbbm{Z}_2$. Then $G$ contains
  $D_4 \times \mathbbm{Z}_2$. As we have seen before, this is impossible.
  
  Thus the assumption $[G : G_1] = 2$ lead to contradiction, $G = G_1 \cong
  A_5 \subset SO (3)$.
  
  This finishes the case $p = 5$.
  
  iii)If $p > 5$, then $L$ contains (as the normalizer of a Sylow-$p$
  subgroup) a group $\mathbbm{Z}_p \rtimes_{\varphi} (\mathbbm{Z}_p^{\ast} /
  \{ \pm 1 \})$ where $\varphi ([a])$ is the multiplication by $a^2$ for $a
  \in \mathbbm{Z}_p^{\ast}$. Let $b$ be a generator of the cyclic group
  $\mathbbm{Z}_p^{\ast}$. Then $\mathbbm{Z}_p \rtimes_{\varphi}
  (\mathbbm{Z}_p^{\ast} / \{ \pm 1 \}) \cong \mathbbm{Z}_p \rtimes_{\phi}
  \mathbbm{Z}_{\frac{p - 1}{2}}$ where $\phi (1)$ is multiplication by $b^2$.
  This group is solvable and contains $\mathbbm{Z}_p, p > 5$, so it is cyclic
  or dihedral (Corollary 3). The element $(0, 1) \in \mathbbm{Z}_p
  \rtimes_{\phi} \mathbbm{Z}_{\frac{p - 1}{2}}$ has order $\frac{p - 1}{2} >
  2$, thus the conjugation of $(0, 1)$ in $\mathbbm{Z}_p$ has to be trivial in
  either case (in the dihedral case $\mathbbm{Z}_p$ is in the unique cyclic
  subgroup of index 2, so does $(0, 1)$). Thus $b^2 = 1 \in \mathbbm{Z}_p^{}$
  which implies $| \mathbbm{Z}_p^{\ast} | = 1$ or 2, consequently $p = 2$ or
  3; a contradiction.
  
  In conclusion $G \subset SO (3)$.
\end{proof}

\

\section{Remarks}

This short section contains the following two results (cf. [16], [21]):

{\tmstrong{Remark 3.1}}: {\tmem{Let $G$ be a finite group acting (orientation
preservingly), locally linearly or smoothly on $\mathbbm{R}^4$. Then $G$ is
isomorphic to a subgroup of $O_4$ ($SO (4)$).}}

{\tmstrong{Remark 3.2}}: {\tmem{There are finite groups $G$ which act topologically
and orientation preservingly on $\mathbbm{R}^4$ and $G\nsubseteq SO (4)$ (in fact
$G\nsubseteq O (7)$).}}

These two remarks are included for the completeness reasons. Our proof of
Remark 3.1 is quite different from the one given in [16]. It is much shorter
and is very much in the spirit of considerations in our proof of Theorem 1.

{\tmstrong{Proof of 3.1}}: Our first observation is the following:

{\tmstrong{FACT}}: The only finite simple group which acts orientation
preservingly on $\mathbbm{R}^4$ is $A_5$.

The above fact follows from the direct inspection of all finite non-abelian
simple groups, in the atlas of \ Finite Groups (cf. [7]). The point here is
that each such group except $A_5$ has a solvable subgroup too large to act
effectively on $\mathbbm{R}^4$. (We recall that each solvable group acting on
$\mathbbm{R}^4$ always has a fixed point). For example, in $A_6$ one can take
the normalizers of Sylow-3 subgroups.

Suppose now that $G$ is NOT simple. Let $H \neq \{ 0 \}$ be a maximal normal
proper subgroup of $G$ (i.e. $G / H$ is simple).

{\underline{Case 1}}: $H$ is a non-abelian simple group (hence $H \cong A_5$).
Then we have an extension
\begin{equation}
  0 \longrightarrow H \longrightarrow G \longrightarrow G / H \longrightarrow
  0
\end{equation}
In order to classify such extensions (cf. [5] p.105), let $\psi : G / H
\rightarrow \tmop{Out} (A_5) \cong \mathbbm{Z}_2$ be a homomorphism. Then
$\psi : G / H \rightarrow \mathbbm{Z}_2$ is trivial except when $G / H \cong
\mathbbm{Z}_2$. (note that $G / H$ is a simple group).

The set of extensions (1) with fixed $\psi$ is classified by $H^2 (G / H ; Z
(H))$, where $Z (H)$ is the center of $H$ (cf. [5] p.105). Consequently there
is only one extension $G \cong H \times G / H$ for $G / H \neq \mathbbm{Z}_2$
and two extensions $G \cong H \times \mathbbm{Z}_2$, and $G \cong S_5$ for $G
/ H \cong \mathbbm{Z}_2$.

It follows that both $G \cong A_5 \times \mathbbm{Z}_2$ and $S_5$ are subgroups of $SO(4)$ and it is not difficult to see that $G \cong A_5 \times G/H$, $G/H \neq \mathbbm{Z}_2$ cannot act on $\mathbbm{R}^4$ 

{\underline{Case 2}}: $H$ is simple abelian. In this case $(\mathbbm{R}^4)^G =
((\mathbbm{R}^4)^H)^{G / H} = \{ \tmop{pt} \}$. Consequently $G \subset
SO (4)$

{\underline{Case 3}}: $H$ is not simple. Repeating the argument from Case 1
and Case 2 with $H$ replacing $G$ one easily concludes $G \subset SO
(4)$.

{\underline{Case 4}}: Suppose $G$ has an orientation reversing element. Let $K \triangleleft G$ be the normal subgroup of orientation preserving elements, so that $G/K \cong \mathbbm{Z}_2$. Then either $(\mathbbm{R}^4)^K \neq \emptyset$ and hence $(\mathbbm{R}^4)^G \neq \emptyset$ and consequently $G \subseteq O(4)$ or $K \cong A_5$ and hence we have an extension
\begin{equation}
  0 \longrightarrow K \longrightarrow G \longrightarrow \mathbbm{Z}_2 \longrightarrow
  0
\end{equation}
This however was handled in Case 1 and hence the proof of Remark 3.1 is concluded.
\

\

{\tmstrong{Proof of 3.2}}: Let $Q (8 p, q)$ be the generalized quaternionic
group (cf. [8]), given by the extension
\[ 0 \longrightarrow \mathbbm{Z}_p \times \mathbbm{Z}_q \longrightarrow Q (8
   p, q) \longrightarrow Q (8) \longrightarrow 0 \]
where $Q (8)$ is the standard quaternionic group of order eight.

It turns out that $\pi = Q (8 p, q)$ i a 4-periodic group and hence acts
freely on a simply connected CW complex $\widetilde{X}$ with $\widetilde{X} \simeq
S^3$. (cf. [8]).

There are conditions on $p, q$ (cf. [8]) (for example $(p, q) = (3, 313), (3,
433), (7, 113), (5, 461), \ldots$) which imply that $\pi$ acts freely on a
closed 3-manifold $\mathcal{M}^3$ which is a homology 3-sphere. Moreover there
is a $\mathbbm{Z} [\pi]$-homology equivalence
\[ k : \mathcal{M}^3 \longrightarrow X = \widetilde{X} / \pi \]
Consider the map $h = k \times \tmop{id} : \mathcal{M}^3 \times I
\longrightarrow X \times I$ and let $\lambda (h) \in L_0^h (\pi)$ be the
surgery obstruction for changing $h$ to a homotopy equivalence without
modifying anything on the boundaries.

Now let $\mathcal{F}: \mathbbm{Z} [\pi] \rightarrow \mathbbm{Z} [\pi]$ be the
identity homomorphism and $\Gamma_0 (\mathcal{F})$ be the Cappell-Shaneson
homological surgery obstruction group as in [6].

The natural homomorphism $j_{\ast} : L_0^h (\pi) \rightarrow \Gamma_0
(\mathcal{F})$ is an isomorphism (see [6] p.288) and clearly $j_{\ast}
(\lambda (h)) = 0$ so that $\lambda (h) = 0$ in $L_0^h (\pi)$.

Let $\bar{h} : (\mathcal{W}^4 ; \mathcal{M}^3, \mathcal{M}^3) \longrightarrow
(X \times I ; X, X)$ be a homotopy equivalence. Form a two ended open manifold
\[ \mathcal{W}^4_0 \assign \ldots \cup \mathcal{W}^4
   \underset{\mathcal{M}^3}{\cup} \mathcal{W}^4 \underset{\mathcal{M}^3}{\cup}
   \mathcal{W}^4 \cup \ldots \]
by stacking together copies of $\mathcal{W}^4$.

Observe that $\pi_1 (\mathcal{W}^4) \cong \pi$ and the universal cover
$\widetilde{\mathcal{W}^4_0}$ of $\mathcal{W}^4_0$ is a manifold properly
homotopy equivalent to $S^3 \times \mathbbm{R}$ and hence homeomorphic to $S^3
\times \mathbbm{R}$ by [13].

One point compactification of one end of $\widetilde{\mathcal{W}^4_0}$ yields
an action of $\pi$ on $\mathbbm{R}^4$ with one fixed point. Since $\pi$ is not
isomorphic to a subgroup of $O (7)$ (cf. [1]) the proof of 3.2 is complete.
\
\end{large}
\
\

\
\\
\\
\textbf{{\LARGE Acknowledgments.}}: We would like to thank Prof. Reinhard Schultz for turning our attention to the problem studied in this paper and suggesting a possible line of attack.
\\
\\
\
\
\
\\\title{\textbf{{\LARGE Reference}}}
\
\begin{flushleft}
[1] S. Bentzen and J. Madsen, On the Swan subgroups of certain periodic groups,
Math. Ann. 162 (1983), 447-474.

[2] R. H. Bing, A homeomorphism between the 3-sphere and sum of two solid
horned spheres, Ann. of Math. 56 (1952), 354-362.

[3] R. H. Bing, Inequivalent families of periodic homeomorphisms of $E^3$, Ann.
of Math. 80 (1964), 78-93.

[4] L. E. J. Brouwer, {\"U}ber die periodischen~Transformationen der Kugel,
Math. Ann. 90 (1919), 39-41.

[5] K. Brown, Cohomology of Groups, Grad. Texts in Math. 87, Springer-Verlag,
New York, Berlin, Heidelberg 1982.

[6] S. Cappell and J. Shaneson, The codimension two placement problem and
homology equivalent manifolds, Ann. of Math. 99 (1974) 277-348.

[7] J. H. Conway, R. T. Curtis, S. P. Norton, R. A. Parker and R. A. Wilson,
Atlas of Finite Groups, Oxford Univ. Press, 1985.

[8] J. Davis and R. Milgram, A survey of the Spherical Space Form Problem,
Math. Reports, Vol. 2, Harvard Academic Publishers. London, 1985, pp. 223-283.

[9] R. Dotzel and G. Hamrick, $p$-group actions on homology spheres, Invent.
Math. 62 (1981), 437-442

[10] A. Edmonds, A topological proof of the equivariant Dehn Lemma, Trans.
Amer. Math. Soc. 297 (1986), 605-615.

[11] S. Eilenberg, Sur les transformations p{\'e}riodiques de la surface de
sph{\`e}re, Fund. Math 22 (1934), 28-41.

[12] W. Feit and J. Thompson, Solvability of groups of odd order, Pacific J.
Math. 13 (1963), 775-102.

[13] M. Freedman, The topology of four-dimensional manifolds, J. Diff. Geom. 17
(1982), 357-453.

[14] C. H. Giffen, The generalized Smith conjecture, Amer. J. Math. 88 (1966),
187-198.

[15] C. McA. Gordon, On the higher dimensional Smith conjecture, Proc. London
Math. Soc.(3), 29(1974), 98-110.

[16] A. Guazzi, M. Mecchia and B. Zimmermann, On finite groups acting on
acyclic low dimensional manifolds, Fund. Math. 215 (2011), 203-217.

[17] M. Hall, The Theory of Groups, The Macmillan Company, New York, 1962.

[18] W. Jaco and H. Rubinstein, PL Equivariant Surgery and Invariant
Decompositions of 3-manifolds, Adv. Math. 73 (1989), 149-191.

[19] B. Kerekjarto, {\"U}ber~die-periodischen~Transformationen der Kreisscheibe
und der Kugelfl{\"a}che, Math. Ann. 80 (1919), 36--38.

[20] S. Kwasik and R. Schultz, Icosahedral group actions on $\mathbbm{R}^3$,
Invent. Math. 108 (1992), 385-402.

[21] S. Kwasik and R. Schultz, Desuspension of group actions and the ribbon
theorem, Topology Vol. 27 (1988), 443-457.

[22] S. Kwasik and R. Schultz, Pseudofree group actions on $S^4$, Amer. J. Math
112 (1990), 47-70.

[23] W. Meeks and S. T. Yau, The equivariant Dehn's Lemma and Loop theorem,
Comment. Math. Helv. 56 (1981), 225-239.

[24] J. Nielsen, Abbildungsklassen endlicher Ordnung, Acta Math. 75 (1942),
24-115.

[25] P. A. Smith, Transformations of finite period, Ann. of Math. (2) 39
(1938), 127-164.

[26] E. Stein, Surgery on products with finite fundamental group, Topology 16 (1977), 473-493

[27] F. Sun, Topological Symmetries of $\mathbbm{R}^3$, II, Preprint, Tulane University, 2016.

[28] M. Suzuki, On finite groups with cyclic Sylow subgroups for all odd
primes, Amer. J. Math. 77 (1955), 657-691.

\end{flushleft}

\
\\
\

Slawomir Kwasik

Department of Math.

Tulane University

New Orleans, LA70118

kwasik@tulane.edu

\

Fang Sun

Department of Math.

Tulane University

New Orleans, LA70118

fsun@tulane.edu

\end{document}